\documentclass[12pt]{amsart}
\usepackage[colorlinks=true,citecolor=black,linkcolor=black,urlcolor=blue]{hyperref}
\usepackage{amsmath,cleveref,mathtools}
\usepackage[english,  activeacute]{babel}
\usepackage[utf8]{inputenc}
\usepackage{amssymb}
\usepackage{amsthm}
\makeatletter
\def\l@subsection{\@tocline{2}{0pt}{2.5pc}{5pc}{}}
\def\l@subsubsection{\@tocline{2}{0pt}{5pc}{7.5pc}{}}
\makeatother
\usepackage{graphics,graphicx}
\usepackage{enumerate}
\usepackage{array}
\usepackage{bm}
\usepackage{a4wide}
\usepackage{color, url}
\usepackage{mathrsfs}
\usepackage{comment}
\usepackage{xcolor}
\usepackage{relsize}
\usepackage{tikz}
\usetikzlibrary{cd}
\usetikzlibrary{matrix,arrows}
\usepackage[dvipsnames]{xcolor}

\usepackage{mmacells}
\usepackage{mma} 

\newcommand{\seqnum}[1]{\href{http://oeis.org/#1}{\underline{#1}}}

\theoremstyle{plain}
\newtheorem{theorem}{Theorem}[section]

\newtheorem{corollary}[theorem]{Corollary}
\newtheorem{lemma}[theorem]{Lemma}

\theoremstyle{definition}
\newtheorem{definition}[theorem]{Definition}

\newtheorem{remark}[theorem]{Remark}

\newcommand{\PF}{\mathrm{PF}}
\newcommand{\FR}{\mathrm{FR}}

\newcommand{\fr}{\mathrm{FR}}
\newcommand{\UFR}{\mathrm{UFR}}
\newcommand{\UPF}{\mathrm{UPF}}
\newcommand{\Lucky}{\mathsf{Lucky}} 
\newcommand{\lucky}{\mathsf{lucky}}

\newcommand{\Fub}{\mathrm{Fub}}

\newcommand{\N}{{\mathbb N}}

\def\modd#1 #2{#1\ \mbox{\rm (mod}\ #2\mbox{\rm )}}

\newcommand{\floor}[1]{\lfloor #1 \rfloor}
\newcommand{\ceil}[1]{\lceil #1 \rceil}

\usepackage[usestackEOL]{stackengine}[2013-10-15]
\stackMath

\title{Lucky Cars in Fubini Rankings and Unit Fubini Rankings}


\author[C. Barreto]{Camilo Barreto}

\author[Beerbower]{Melissa Beerbower}

\author[Elder]{Jennifer Elder}
\address[J.~Elder]{Department of Computer Science, Mathematics and Physics, Missouri Western State University, St. Joseph, MO 64507}
 \email{\textcolor{blue}{\href{mailto:jelder8@missouriwestern.edu}{jelder8@missouriwestern.edu}}}
 
\author[P.~E.~Harris]{Pamela E. Harris}
\address[M. ~Beerbower, P.~E.~Harris]{Department of Mathematical Sciences, University of Wisconsin-Milwaukee, Milwaukee, WI 53211 United States}
\email{\textcolor{blue}{\href{mailto:mbeerbower@luc.edu}{beerbow2@uwm.edu}}, \textcolor{blue}{\href{mailto:peharris@uwm.edu}{peharris@uwm.edu}}}

\author[Martinez]{Lucy Martinez}
\address[L. Martinez]{Department of Mathematics, Rutgers University, Piscataway, NJ 08854}
\email{\textcolor{blue}{\href{mailto:lucy.martinez@rutgers.edu}{lucy.martinez@rutgers.edu}}}

\author[J. L. Ram\'{\i}rez]{Jos\'e L. Ram\'{\i}rez}

\author[S. Ram\'{\i}rez]{Samuel Ram\'{\i}rez}

\author[Shirley]{Grant Shirley}
\address[G. ~Shirley]{Department of Mathematics and Statistics, East Tennessee State University, Johnson City, TN 37614}
\email{\textcolor{blue}{\href{mailto:shirleyg@etsu.edu}{shirleyg@etsu.edu}}}

\author[J. C. V\'{a}squez]{Julio C.  V\'{a}squez}
\address[C. Barreto, J. L. Ram\'{\i}rez, S. Ram\'{\i}rez, J. C. V\'{a}squez]{Departamento de Matem\'aticas,  Universidad Nacional de Colombia,  Bogot\'a, Colombia}
\email{\textcolor{blue}{\href{mailto:cbarreto@unal.edu.co}{cbarreto@unal.edu.co}}}
\email{\textcolor{blue}{\href{mailto:jlramirezr@unal.edu.co}{jlramirezr@unal.edu.co}}}
\email{\textcolor{blue}{\href{mailto:samramirezra@unal.edu.co}{samramirezra@unal.edu.co}}}
\email{\textcolor{blue}{\href{mailto:julcvasquezare@unal.edu.co}{julcvasquezare@unal.edu.co}}}

\usepackage{fancyhdr}

\begin{document}

\begin{abstract}
We study \emph{lucky cars} in subsets of parking functions, called Fubini rankings and unit Fubini rankings. A Fubini ranking is a sequence of nonnegative integers that encodes a valid ranking of competitors, where ties are allowed. A car (or competitor) is said to be \emph{lucky} if it is the first instance of that rank appearing in the sequence.  We present combinatorial characterizations and enumeration formulas for lucky cars in both Fubini rankings and unit Fubini rankings, and establish connections between these objects and ordered set partitions, as well as integer compositions.  
To obtain our results, we use several techniques to enumerate statistics over these families of objects. 
In particular, we employ generating functions, bijective and combinatorial arguments, recurrence relations, and Zeilberger’s creative telescoping method.
\end{abstract}
\maketitle

\pagestyle{fancy}
\fancyhead{} 
\renewcommand{\headrulewidth}{0pt}
\fancyhead[CE]{Lucky Cars of Fubini Rankings and Unit Fubini Rankings}
\fancyhead[CO]{Barreto et al.}
\fancyfoot{} 
\fancyfoot[C]{\thepage}

\section{Introduction}

Throughout, we let $\N=\{1,2,3,\ldots\}$, and for $n\in \N$ define $[n]=\{1,2,\ldots,n\}$. 
Recall that the set of parking functions of length $n$, denoted $\PF_n$, is the set of all $n$-tuples $\alpha=(a_1,a_2,\ldots,a_n)\in[n]^n$ whose weakly increasing rearrangement, denoted  $\alpha^\uparrow=(a_1',a_2',\ldots,a_n')$, satisfy $a_i'\leq i$ for all $i\in[n]$.
Konheim and Weiss, and independently by Pyke, established that $|\PF_n|=(n+1)^{n-1}$, see \cite{konheim1966occupancy,Pyke}.

Another way to describe parking functions is to consider the tuples in the context of cars trying to park on a one-way street. An $n$-tuple $(a_1,a_2,\ldots,a_n)\in[n]^n$ is considered as a list of parking preferences where $a_i$ is the preferred parking spot of the $i$th car. Each car enters the one-way street one at a time, where the parking spots are numbered consecutively, and attempts to park using the following parking rule: 
\begin{itemize}
\item For each $i=1,2,\ldots, n$, car $i$ with preference $a_i$ drives to that spot and parks if it finds that spot available.
\item If spot $a_i$ is unavailable, car $i$ continues down the street and parks in the first available parking spot it encounters. 
\end{itemize}
If all cars are able to park on the street using this procedure, the list of preferences is a parking function. For example, the parking function $(3,2,4,4,1)$ parks the cars in the order car 5, car 2, car 1, car 3 and car 4. In comparison, the preference tuple $(1,5,5,2,3)$ fails to be a parking function as car 3 finds its preferred parking spot occupied and there are no further spots on the street in which it can park.

After defining a set of combinatorial objects, one natural direction of study is to consider discrete statistics and statistic generating functions on the set. 
In the realm of parking functions and subsets of parking functions, the \emph{lucky statistic} is usually one of the first statistics to consider.
Studied by Gessel and Seo~\cite{GesselandSeo}, the lucky statistic counts the number of cars that are able to park in their preferred spot, making these cars \emph{lucky}. 
We let $\Lucky(\alpha)$ denote the set of lucky cars in the parking function $\alpha$, and we let $\lucky(\alpha)=|\Lucky(\alpha)|$, denote the number of lucky cars in $\alpha$. 

For example, for the parking function $\alpha_1=(3,2,4,4,1)$, cars $1, 2$, $3$, and $5$ are lucky, so $\Lucky(\alpha_1)=\{1,2,3,5\}$ and $\lucky(\alpha_1)=4$.  If $\alpha_2=(1,1,\ldots,1)\in[n]^n$, then  $\Lucky(\alpha_2)=\{1\}$, which is the ``unluckiest'' parking functions since the only car that parks in its preferred spot is car $1$. 
Permutations of $[n]$ are the ``luckiest'' parking functions, as every car gets to park in their preference; that is, if $\alpha$ is a permutation of $[n]$, then $\Lucky(\alpha)=[n]$ and $\lucky(\alpha)=n$.

Gessel and Seo also gave the following generating function for the lucky car statistic on parking functions~\cite{GesselandSeo}:
\begin{align}\label{eq:GesselSeo}
  L_n(q)=\displaystyle\sum_{\alpha \in \PF_{n} }q^{\lucky(\alpha)} = q\prod_{i=1}^{n-1} (i+(n-i+1)q),  
\end{align}
and we refer to $L_n(q)$ as the \textit{lucky polynomial} of $\PF_n$.
Their work is related to rooted trees, and the proof relies on ordered set partitions and generating function techniques.  Slav\'{i}k and Vestinick\'{a} in \cite[Equation 11]{slavikvestenicka} provide an analogous result for \eqref{eq:GesselSeo} in the case where cars have different sizes. 

Other work on the lucky statistic on parking functions includes that of Aguillon et al.
~\cite{displacement_hanoi}, which studies the subset of parking functions with $n-1$ lucky cars in which the single unlucky car parks precisely one spot away from its preference. 
They establish a bijection between this subset of parking functions and ideal states in the tower of Hanoi game \cite[Theorem 1]{displacement_hanoi}.
Harris, Kretschmann, and Mart\'inez Mori \cite{harris2023lucky}, consider the sequences $\alpha \in [n]^n$
(not necessarily parking functions) with exactly $n-1$ lucky cars.
They establish that the number of these sequences is equal to the total number of comparisons performed by the \texttt{Quicksort} algorithm given all possible orderings of an array of size $n$~\cite{harris2023lucky}. Colmenarejo et al.~\cite{colmenarejo2024luckydisplacementstatisticsstirling} consider the set of Stirling permutations of order $n$, which are a subset of parking functions of length~$2n$. 
Among, their results they establish that the number of Stirling permutations with maximally many lucky cars (which is $n$ lucky cars) is a Catalan number, 
and that there are $(n-1)!$ Stirling permutations with exactly one lucky car (that being the first car). Also, Harris
and Martinez determined formulas for the number of parking functions with a fixed set of lucky cars \cite{FixedLuckyCars}. They first characterized the outcomes permutations of parking functions with a fixed lucky set, and used the characterization and an argument often referred to as ``counting through permutations'' to enumerate the parking functions with a fixed lucky set.

In this article, we study the properties related to lucky cars and lucky statistics in the following three subsets of $\PF_n$: 

\begin{enumerate}
    \item[($\FR_n$):] The Fubini rankings with $n$ competitors are  $n$-tuples $\alpha=(a_1,a_2,\ldots,a_n)\in [n]^n$ that records a valid ranking of $n$ competitors with ties allowed (\Cref{def:fubini_rank}) and the set is denoted by $\FR_n$. We let $\FR_{n}^\uparrow$ denote the set of weakly increasing Fubini rankings with $n$ competitors.

\item[($\UPF_n$):] Unit interval parking functions of length $n$ are parking functions of length $n$ in which cars park in their preference or in the spot in front of their preference (\Cref{def:unit interval pfs}). We let $\UPF_n$ denote the set of unit interval parking functions of length~$n$. 
Hadaway \cite{Hadaway_unit_interval} established a bijection between $\UPF_n$  and $\FR_n$. 
We let 
    $\UPF_n^{\uparrow}$ denote the set of weakly increasing unit interval parking functions.
    
    \item[($\UFR_n$):] The set of unit Fubini rankings with $n$ competitors consists of the intersection between the set of unit interval parking functions and Fubini rankings, and consists of Fubini rankings in which a competitor can tie with at most one other competitor at that rank (\Cref{def: unit fubini rankings}). 
    We denote this set by $\UFR_n$ and let $\UFR_n^\uparrow$ denote the subset consisting of weakly increasing unit Fubini rankings.
    \end{enumerate}

\Cref{secA} provides results related to Fubini rankings, followed by the results for weakly increasing Fubini rankings. These results are also true for the unit interval parking functions through the bijective map in \cite[Theorem 2.5]{unit_pf}. \Cref{secB} provides results related to unit Fubini rankings, followed by the results for weakly increasing unit Fubini rankings. 
Sections \ref{secA} and \ref{sec:weakly increasing fubinis}, and Sections \ref{secB} and \ref{sec:w_inc_ufr} are organized as follows:
\begin{enumerate}
\item Definitions and background on the sets of interest.
\item Characterization of the set of lucky cars given an element in the set. We call these \textit{lucky sets}.
\item Set enumerations through the number of lucky cars, and related recursions.
\item The lucky polynomials over the set of interest.
\item Results on expected number of lucky cars in an element in the set.
\item Exponential generating functions for the lucky statistic over the set, and exponential generating functions for the size of the set.
\end{enumerate}

Then Sections \ref{sec:fixed lucky cars} and \ref{sec:weakly increasing fubinis lucky sets}, and Sections \ref{sec:unit fubini lucky sets} and \ref{sec:weakly increasing unit fubini lucky sets} focus on enumerating the number of elements contained in each of the sets of interest through fixing sets of lucky cars.
We conclude in \Cref{sec:future} with some directions for future study.

\section{Enumeration of Fubini Rankings by the  Number of Lucky Cars}\label{secA}

In this section, we enumerate  Fubini rankings and weakly increasing Fubini rankings according to the number of lucky cars.  We begin by recalling the following definition (see~\cite{boolean}).

\begin{definition}\label{def:fubini_rank}
A \emph{Fubini ranking with $n$ competitors} is a tuple $\alpha = (a_1, a_2, \dots, a_n)\in [n]^n$ that encodes a valid ranking of $n$ competitors, where ties are allowed; that is, multiple competitors can be tied and have the same rank.  If $k$ competitors share rank $i$, then the subsequent $k-1$ ranks $i+1, i+2, \ldots, i+k-1$ are omitted.
\end{definition}

For example, the tuple $(8,1,2,5,6,2,2,6)$ is a Fubini ranking with $8$ competitors: competitor 1 ranks eighth, competitor 2 ranks first, competitors 3, 6, and 7 tie at rank two, competitor 4 ranks fifth, and competitors 5 and 8 rank sixth.  However, $(1,1,1,2)$ is not a Fubini ranking, since the triple tie for rank one would disallow ranks two and three, making rank four the next available rank. 

\begin{remark}
A Fubini ranking $\alpha=(a_1,a_2,\ldots,a_n)$ can be interpreted as the outcome of a race among $n$ competitors (cars), where ties are allowed. Each entry $a_i$ represents the final position of car~$i$, and all cars with the same value of $a_i$ form a \emph{tie block} at that position. 
Under this interpretation, each Fubini ranking records the order in which groups of cars finish, allowing for ties, while preserving the order of appearance of new ranks.
\end{remark}

We let $\fr_n$ denote the set of Fubini rankings with $n$ competitors. 
For example, the 13 Fubini rankings with $3$ competitors are
\[\fr_3=\left \{\begin{matrix}(1,1,1),(1,1,3),(1,3,1),(3,1,1),(1,2,2),(2,1,2),(2,2,1),\\(1,2,3),(1,3,2),(2,1,3),(2,3,1),(3,1,2),(3,2,1)\end{matrix}\right\}.\]
We remark that 
Cayley~\cite{cayley_2009} proved that $|\fr_n|=\Fub_n$, where
$\Fub_n$ denotes the $n$th Fubini number (OEIS \seqnum{A000670} \cite{OEIS}), which are also known as the ordered Bell numbers. These numbers are given by the expression
\begin{align}\label{eq:fubini numbers2}
 \Fub_{n}
  =\sum_{k=1}^n k!\, S(n,k),
\end{align}
where $S(n,k)$ are the Stirling numbers of the second kind (OEIS \seqnum{A008277} \cite{OEIS}) and count the number of set partitions of $[n]$ into $k$ nonempty parts. Moreover, Fubini rankings are permutation invariant, which means that any rearrangement of the entries of a Fubini ranking is again a Fubini ranking.

We recall that every Fubini ranking is a parking function (see \cite[Lemma 2.3]{unit_pf}), hence we can study the lucky cars of Fubini rankings.
Next we establish that, in a Fubini ranking, the set of lucky cars consists of the first cars to place in a new rank, that is, the first car in each tie block.  This property does not hold for parking functions:  for example, in the parking function $(1,1,2)$, the only lucky car is the first one, even though car~3 is the first to prefer spot~2.

\begin{theorem}
\label{teor:lucky count is equal to number of ranks}
Let $\alpha=(a_1,a_2,\ldots,a_n)\in\FR_n$ and let  $1=r_1<r_2<\cdots<r_k$ be the distinct ranks appearing in $\alpha$. Then 
\[\Lucky(\alpha)=\{i\in[n]:a_i\neq a_j \mbox{ for all $j<i$}\}.\] 
Moreover,  $\lucky(\alpha)=k$, which is the number of distinct ranks appearing in $\alpha$.
\end{theorem}

\begin{proof}
Let $\alpha \in \FR_n$ be a Fubini ranking with $n$ competitors and with $k$ distinct ranks  $1=r_1<r_2<\cdots <r_k$.  Suppose that $c_j$ cars have rank $r_j$, so that $\sum_{j=1}^k c_j=n$.  
    
The first car is always lucky. Since $c_1$ cars have rank $r_1=1$, the next available rank is $r_2=c_1+1$, which is shared by $c_2$ cars.  In general, for each $1 \leq j \leq k$,  there are $c_j$ cars with  rank  $r_j=c_1+\cdots+c_{j-1}+1$.  

Thus, the Fubini ranking can be divided into $k$ consecutive tie blocks:  
block~1 contains the $c_1$ cars with rank~$r_1$ (where $r_1=1$); block~2 contains the next $c_2$ cars with rank~$r_2$; and so on, until block~$k$, which contains the $c_k$ cars with rank~$r_k$.

The first car in each block is lucky, since it is the first to receive that rank, while the remaining cars in the same tie block share the rank and find that the first car (in that tie block) has already occupied it.  Therefore, a car in a Fubini ranking is lucky if and only if it is the first car in its tie block.  
Consequently, any Fubini ranking with $k$ distinct ranks has exactly $k$ lucky cars.
\end{proof}

From the previous proof, it follows that there is a natural correspondence between Fubini rankings and ordered set partitions.

 \begin{definition}
An \emph{ordered set partition} of $[n]$ into $k$ blocks is a sequence of disjoint nonempty subsets
$(B_1, B_2, \ldots, B_k)$ whose union is $[n]$.  Unlike in an ordinary (unordered) set partition, the order of the blocks matters. 
 \end{definition}

For completeness, and to make the exposition self-contained, we now describe explicitly the bijection between Fubini rankings and ordered set partitions (cf.~\cite{BHRRV}).\label{BijFR}

Let $\alpha=(a_1,a_2,\dots,a_n)$ be a Fubini ranking with $n$ competitors and with $k$ distinct ranks. 
Define
\[ B_1= \left\{ j : a_j = 1 \right\} \text{ and }
   B_i = \left\{ j : a_j = 1 + \sum_{\ell=1}^{i-1} |B_\ell| \right\}, 
   \quad \text{for } i=2,3,\dots,k.
\]
Each $B_i$ is nonempty, the sets are pairwise disjoint, and together they cover $[n]$, so $(B_1,B_2,\dots,B_k)$ is an ordered set partition of $[n]$.

Conversely, given an ordered set partition $(B_1,B_2,\dots,B_k)$ of $[n]$, the corresponding Fubini ranking $\alpha=(a_1,a_2,\dots,a_n)$ is recovered by setting
\[
a_i = 1 + \sum_{\ell=1}^{j-1}|B_\ell| 
\qquad \text{whenever } i\in B_j.
\]
For example, if $\alpha=(8,1,2,5,6,2,2,6)\in\FR_8$, then the corresponding ordered set partition is $(\{2\},\{3,6,7\},\{4\},\{5,8\},\{1\})$. 
Under this bijection, the number of blocks equals the number of distinct ranks, and therefore equals the number of lucky cars.
 
\begin{corollary}
\label{coro: k lucky fub}
    Let $f_\FR(n,k)$ denote the number of Fubini rankings with $n$ competitors and $k$  lucky cars. Then $f_\FR(n,k)=k! S(n,k)$.
\end{corollary}

From \cite[Corollary~2.13]{unit_pf} we deduce the following identity for $f_\FR(n,k)$. For completeness, we provide here a direct combinatorial proof.

\begin{theorem}\label{formula1FR}
Let $f_\FR(n,k)$ denote the number of Fubini rankings with $n$ competitors and $k$ lucky cars. Then
    $$f_\FR(n,k) =\sum_{(c_1,c_2,\ldots , c_k ) \vDash n} \binom{n}{c_1, c_2, \ldots , c_k},$$
    where the sum ranges over all compositions of $n$ into $k$ positive parts.
\end{theorem}

\begin{proof}
Fubini rankings of size~$n$ with $k$ lucky cars can be counted using the following procedure.   Let $r_1< r_2< \dots< r_k$ be the distinct ranks appearing in a Fubini ranking with $n$ competitors, and let $c_i$ $(1 \leq i \leq k)$ denote the number of cars sharing rank $r_i$. Hence,  $\sum_{j=1}^k c_j= n$.  

We proceed as follows. From the $n$ available positions, first choose $c_1$ of them to assign the rank $r_1=1$;  then choose $c_2$ positions to assign the rank $r_2 = 1 + c_1$ and  continue in this way for each subsequent rank.  The number of ways to carry out this process is given by the multinomial coefficient  $\binom{n}{c_1, c_2, \ldots, c_k}$.

Finally, summing over all ordered $k$-tuples $(c_1,c_2,\ldots,c_k)$ of positive integers with total~$n$, that is, over all compositions $(c_1,\ldots,c_k)\vDash n$,
yields the desired identity.
\end{proof}

\begin{theorem}\label{recSec1}
For $n, k\geq 1$, we have 
\[f_\FR(n,k)=k(f_\FR(n-1,k) + f_\FR(n-1,k-1)).\]
\end{theorem}
\begin{proof}
Let $\alpha$ be a Fubini ranking of length~$n$ with $k$ lucky cars.  Consider the last car, that is, the $n$th car. There are two possibilities.  If the last car is tied with at least one other car, then the first $n-1$ cars form a Fubini ranking with the same number $k$ of distinct ranks, giving $f_{\FR}(n-1,k)$ possibilities. Since the last car can join any of the $k$ existing ranks, this case contributes $k f_{\FR}(n-1,k)$ possibilities.  

Otherwise, if the last car is not tied with any other car, then the remaining $n-1$ cars must form a Fubini ranking with $k-1$ distinct ranks, which can be done in $f_{\FR}(n-1,k-1)$ ways. In this situation, the last car can take any of the $k$ possible ranks, contributing $k f_{\FR}(n-1,k-1)$ possibilities.  
Adding both cases, we obtain the desired recurrence relation.
\end{proof}

Using the recurrence relation in Theorem~\ref{recSec1}, we can compute the first few values of the sequence $f_{\FR}(n,k)$:
$$[f_\FR(n,k) ]_{n, k\geq 0}=
\begin{pmatrix}
1 & 0 & 0 & 0 & 0 & 0 & 0 & 0 \\
 0 & 1 & 0 & 0 & 0 & 0 & 0 & 0 \\
 0 & 1 & 2 & 0 & 0 & 0 & 0 & 0 \\
 0 & 1 & \framebox{\textbf{6}}  & 6 & 0 & 0 & 0 & 0 \\
 0 & 1 & 14 & 36 & 24 & 0 & 0 & 0 \\
 0 & 1 & 30 & 150 & 240 & 120 & 0 & 0 \\
 0 & 1 & 62 & 540 & 1560 & 1800 & 720 & 0 \\
 0 & 1 & 126 & 1806 & 8400 & 16800 & 15120 & 5040 \\
\end{pmatrix}.$$  

This array matches with OEIS \seqnum{A019538} \cite{OEIS}. For example, the boxed entry corresponds to $f_\FR(3,2)=6$.  The six Fubini rankings with  $3$ competitors and with exactly two lucky cars are 
\[
(1,2,2),\ (2,1,2),\ (2,2,1),\ (1,1,3),\ (1,3,1),\ (3,1,1).
\]
Each of these rankings has precisely two lucky cars.

As a consequence of \Cref{coro: k lucky fub}, we obtain a complete description of the lucky polynomial for Fubini rankings.

\begin{corollary}
\label{cor:lucky poly for fubini}
The  lucky polynomial of Fubini rankings with $n$ competitors is given by 
    \[L_{\FR_n}(q)=\sum_{\alpha\in\fr_n}q^{\lucky(\alpha)}=
   \sum_{k=0}^{n}  k!S(n,k)q^k, \quad n\geq 1 
    \]
    and $L_{\FR_0}(q)=1$.
\end{corollary}

For example, the first few lucky polynomials are
\begin{align*}
&L_{\FR_1}(q)=q, \qquad L_{\FR_2}(q)= 2q^2+q, \qquad L_{\FR_3}(q)=6 q^3+6 q^2+q, \\
&L_{\FR_4}(q)=24 q^4+36 q^3+14 q^2+q, \qquad L_{\FR_5}(q)=120 q^5+240 q^4+150 q^3+30 q^2+q.
   \end{align*}

Given the previous results, we can determine the expected number of lucky cars in a Fubini ranking, which coincides with the expected number of distinct ranks.

\begin{corollary}\label{cor:expected number of lucky cars in a Fubini ranking} 
The expected number of lucky cars in a Fubini ranking with $n$ competitors is
\begin{align}\label{eq:average lucky in Fub Rankings}
\mathbb{E}[\lucky(\alpha): \alpha \in \FR_n] =\frac{1}{\Fub_n} \sum_{k=1}^n k\cdot k!\,S(n,k).
\end{align}
Moreover, 
\[\mathbb{E}[\lucky(\alpha): \alpha \in \FR_n]  
 \sim \frac{n}{2 \log 2}. \]
 \end{corollary}
\begin{proof}
   We take the arithmetic mean of number of lucky cars over all Fubini ranking with $n$ competitors. Among the Fubini rankings with exactly $k$ lucky cars, the total contribution to the sum of lucky cars is   $k\cdot  f_\FR(n,k)=k\cdot k!S(n,k)$. Summing over all possible values  $1\leq k\leq n$ gives the total number lucky cars among all Fubini rankings with $n$ competitors. The expectation is then given by taking that sum and dividing by the total number of Fubini rankings, which according to Equation~\eqref{eq:fubini numbers2}, is given by $\Fub_n$. This establishes \Cref{eq:average lucky in Fub Rankings}.

For the asymptotic estimate, we use the asymptotic relations given in OEIS sequences \seqnum{A000670} and \seqnum{A069321}~\cite{OEIS}:
\begin{align}\label{eq:assymptotic estimates}
\Fub_n \sim \frac{n!}{2(\log 2)^{n+1}}
\qquad\text{and}\qquad
\sum_{k=1}^n k\cdot k!\,S(n,k)
   \sim \frac{n!\,n}{4(\log 2)^{n+2}},
\end{align}
where we recall that $f(n)\sim g(n)$ means that $\lim_{n\to\infty}\frac{f(n)}{g(n)}=1$.
Taking the quotient of the two asymptotic forms in \eqref{eq:assymptotic estimates} gives  the desired result.
\end{proof}

\begin{remark}
Observe that $\sum_{k=1}^n k\cdot k!S(n,k)$ corresponds to a Stirling transform; (OEIS \seqnum{A069321} \cite{OEIS}). 
\end{remark}

Next, we give a closed formula for the exponential generating function of $f_{\FR}(n,k)$.

\begin{theorem}\label{egfluckyp}
The exponential generating function for the number of Fubini rankings with $n$ competitors and with $k$ lucky cars is 
  \[\sum_{n\geq 0} \sum_{k\geq 0} f_\FR(n,k)q^k \frac{x^n}{n!}=\frac{1}{1-(e^x-1)q}.
\]
\end{theorem}

\begin{proof}
By  \Cref{coro: k lucky fub}, we have  $f_\FR(n,k) = k!S(n,k)$. The Stirling numbers of the second kind satisfy the exponential generating function (cf.~\cite{Mezo2}) 
\[\sum_{n\geq 0} S(n,k)\frac{x^n}{n!}=\frac{(e^x-1)^k}{k!}.\]
Therefore, 
\begin{align*}
   \sum_{n\geq 0}\sum_{k\geq 0} k!S(n,k)q^k\frac{x^n}{n!} &= \sum_{k\geq 0}k!q^k\left(\sum_{n\geq 0}S(n,k)\frac{x^n}{n!}\right)
    =\sum_{k\geq 0} k!q^k\frac{(e^x-1)^k}{k!}\\
    &=\sum_{k\geq 0}(e^x-1)^kq^k
    =\frac{1}{1-(e^x-1)q},
\end{align*}
where the last equality follows from the geometric series formula.
\end{proof}

As a direct consequence, setting $q=1$ recovers the exponential generating function for the Fubini numbers.

\begin{corollary}\label{coroEGF1}
The exponential generating function for the total number of Fubini rankings with $n$ competitors is
\[
\sum_{n\ge0} \Fub_n\,\frac{x^n}{n!}
   = \frac{1}{2 - e^{x}}.
\]
\end{corollary}

\subsection{Weakly Increasing Fubini Rankings}\label{sec:weakly increasing fubinis}
Let $f_\FR^\uparrow(n,k)$ denote the number of weakly increasing Fubini rankings with $n$ competitors and with $k$ distinct ranks (equivalently, $k$ lucky cars). Our first result gives a simple closed formula for $f_{\FR}^\uparrow(n,k)$. 

\begin{theorem}\label{teor: weakly increasing Fubini}
The number of weakly increasing Fubini rankings with $n$ competitors and with $k$ lucky cars is  
    $f_\FR^{\uparrow}(n,k) = \binom{n-1}{k-1}$, and the total number of weakly increasing Fubini rankings is $|\FR_n^\uparrow|=2^{n-1}$.
\end{theorem}

\begin{proof}
Consider the $n$ cars as labeled  balls placed into $k$ urns, one for each distinct rank used.  The key difference from the general case is that, since the sequence is weakly increasing, the first $c_1$ cars must share the smallest rank, the next $c_2$ cars share the next rank, and so on.  Hence, the weakly increasing condition fixes the order of the cars, and they behave as unlabeled balls being distributed among $k$ ordered urns.

Knowing how many balls are in each urn (with no urn empty) is precisely equivalent to knowing how many Fubini rankings there are with exactly $k$ distinct ranks.    The count for placing $n$ labeled balls among $k$ labeled urns, with no empty urns, is given by the binomial coefficient $\binom{n-1}{k-1}$. Thus $f_\FR^\uparrow(n,k) = \binom{n-1}{k-1}$.
    
Finally, summing over all possible values of $k$ yields
    \[|\FR_n^\uparrow|=\sum_{k=1}^{n}\binom{n-1}{k-1}=2^{n-1},\] 
as claimed.
\end{proof}

\begin{remark}\label{recFRI}
It is immediate from Theorem~\ref{teor: weakly increasing Fubini} that the numbers 
$f_\FR^{\uparrow}(n,k)$ satisfy the recurrence relation
\[
f_\FR^{\uparrow}(n,k)=f_\FR^{\uparrow}(n-1,k)+f_\FR^{\uparrow}(n-1,k-1),
\]
with initial condition $f_\FR^{\uparrow}(1,1)=1$. This follows directly from the binomial identity 
$\binom{n-1}{k-1}=\binom{n-2}{k-1}+\binom{n-2}{k-2}$.
\end{remark}

The first few values of the sequence $f^\uparrow_\FR(n,k)$ are shown below:
$$[f^\uparrow_\FR(n,k) ]_{n, k\geq 0}=
\begin{pmatrix}
 1 & 0 & 0 & 0 & 0 & 0 & 0 & 0 \\
 0 & 1 & 0 & 0 & 0 & 0 & 0 & 0 \\
 0 & 1 & 1 & 0 & 0 & 0 & 0 & 0 \\
 0 & 1 & 2 & 1 & 0 & 0 & 0 & 0 \\
 0 & 1 & 3 & 3 & 1 & 0 & 0 & 0 \\
 0 & 1 & \framebox{\textbf{4}} & 6 & 4 & 1 & 0 & 0 \\
 0 & 1 & 5 & 10 & 10 & 5 & 1 & 0 \\
 0 & 1 & 6 & 15 & 20 & 15 & 6 & 1 \\
\end{pmatrix}.$$  
This array coincides with a shifted version of Pascal’s triangle.  
For example, the boxed entry corresponds to $f^\uparrow_\FR(5,2)=4$.  
The four weakly increasing Fubini rankings with~$5$ competitors and with exactly two lucky cars are
\[
(1,1,1,1,5),\quad 
(1,1,1,4,4),\quad
(1,1,3,3,3),\quad
(1,4,4,4,4).
\]

We now derive a closed formula for the lucky polynomial of  weakly increasing Fubini rankings.  We provide both an algebraic and a combinatorial proof.

\begin{corollary}\label{coro:FRn lucky poly}
    The lucky polynomial of
    weakly increasing Fubini rankings with $n$ competitors is
    \[L_{\FR_n^\uparrow}(q)  =\sum_{\alpha\in \FR^\uparrow_n}q^{\lucky(\alpha)} = q(q + 1)^{n - 1}, \quad n\geq 1,\]
    and $L_{\FR_0^\uparrow}(q)=1$.    
\end{corollary}

\begin{proof}\emph{Algebraic argument.}
    By \Cref{teor: weakly increasing Fubini},  we have
    \[L_{\FR_n^\uparrow}(q) = \sum_{k=1}^n \binom{n-1}{k-1}q^k.\]
Reindexing the sum to start at zero, factoring out $q$, and applying the binomial theorem gives
    \[L_{\FR_n^\uparrow}(q)=\sum_{k=1}^{n} \binom{n-1}{k-1}q^{k}=\sum_{k=0}^{n-1} \binom{n-1}{k}q^{k+1}=q\sum_{k=0}^{n-1} \binom{n-1}{k}q^{k}=q(q+1)^{n-1}.\qedhere\] 
  \end{proof}
  \begin{proof}[Combinatorial argument.]
 Let $\alpha = (a_1, a_2, \dots, a_n)$ be a weakly increasing Fubini ranking with $n$ competitors. 
 Since $1 = a_1 \leq a_2 \leq \cdots \leq a_n$, the first element is always lucky. For $i \geq 2$, the $i$th element is lucky precisely when $a_i > a_{i-1}$, that is, when it is the first  occurrence of a new value. Thus, each new value contributes one lucky car, while repeated values correspond to unlucky cars.  
 
 For  such a Fubini ranking with $n$ competitors, there is exactly one way to place the first lucky car (with weight $q$), and each of the remaining $n-1$ positions can either repeat the previous value (unlucky, weight 1) or introduce a new value (lucky, weight $q$). Hence the generating function is given by $
L_{\FR_n^\uparrow}(q) = q (q+1)^{n-1}$.
\end{proof} 

Given the previous results, we now determine the expected number of lucky cars in a weakly increasing
Fubini ranking.

\begin{corollary}\label{cor:expected number of lucky cars in a weakly increasing Fubini ranking}
    The expected number of lucky cars in a weakly increasing Fubini ranking with $n$ competitors is
    $$\mathbb{E}[\lucky(\alpha):  \alpha\in \FR_n^\uparrow]
    = \frac{n+1}{2}.$$
\end{corollary}
\begin{proof}
    This proof is analogous to that of  \Cref{cor:expected number of lucky cars in a Fubini ranking}.     The expected value is given by
    \begin{align*}
\sum_{k=1}^n \frac{1}{2^{n-1}}k \binom{n-1}{k-1} &=  \frac{1}{2^{n-1}}\sum_{k=0}^{n-1}(k+1) \binom{n-1}{k} \\&= \frac{1}{2^{n-1}}\left(\sum_{k=0}^{n-1} k\binom{n-1}{k} + \sum_{k=0}^n \binom{n-1}{k} \right) \\
    &=\frac{1}{2^{n-1}}\left( (n-1)2^{n-2}+ 2^{n-1}\right)\label{eq:use identities}\\&= \frac{n+1}{2}.
    \end{align*}
In the calculations above we have used the classical binomial identities
\[
\sum_{k=0}^n k\binom{n}{k} = n2^{n-1}
\qquad\text{and}\qquad
\sum_{k=0}^n \binom{n}{k} = 2^{n}.\qedhere
\]
\end{proof}

\begin{corollary}\label{egfluckypw}
The exponential generating function for the number of weakly increasing Fubini rankings with $n$ competitors and with $k$ lucky cars is 
  \[\sum_{n\geq 0} \sum_{k\geq 0} f_\FR^\uparrow(n,k)q^k \frac{x^n}{n!}=\frac{1 + e^{(1 + q) x} q}{1+q}.
\]
\end{corollary}
\begin{proof}
By  \Cref{coro:FRn lucky poly}, we have  
\begin{align*}
 \sum_{n\geq 0} \sum_{k\geq 0} f_\FR^\uparrow(n,k)q^k \frac{x^n}{n!}=1+\sum_{n\geq 1}q(q+1)^{n-1}\frac{x^n}{n!} &=1 + \frac{q}{q+1}(e^{(q+1)x}-1)=\frac{1 + e^{(1 + q) x} q}{1+q},
\end{align*}
where the last equality follows from the exponential generating function of the exponential function $e^x$.
\end{proof}

As a direct consequence, setting $q=1$ recovers the exponential generating function for the number of weakly increasing Fubini rankings.

\begin{corollary}\label{coroEGF1b}
The exponential generating function for the total number of weakly increasing Fubini rankings with $n$ competitors is
\[
\sum_{n\ge0} |\FR_n^\uparrow|\,\frac{x^n}{n!}
   = \frac{1}{2}(1+e^{2x}).
\]
\end{corollary}

\subsection{Fubini Rankings with a Fixed Set of Lucky Cars}\label{sec:fixed lucky cars}
The goal of this section is to determine the cardinality of the set of Fubini rankings with $n$ competitors having a fixed set of lucky cars. 
Throughout we consider $X$ to be a subset of the set of parking functions of length $n$, and let $I \subseteq[n]$ be a fixed subset of cars. 
We say that $I$ is an \textit{$n$-admissible lucky set} (or \textit{lucky set} for short) if there exists $\alpha\in X$ such that $\Lucky(\alpha)=I$.
We then define  
\[\Lucky_X(I)\coloneq\{\alpha\in X~:~\Lucky(\alpha)=I\}.\]
In other words, $\Lucky_X(I)$ denotes the set of parking functions in $X$ of length $n$ whose lucky cars are precisely those in the specified indexing set $I$.

Harris and Martinez \cite{FixedLuckyCars}  characterized the set of parking outcomes arising from parking functions with a fixed set of lucky cars and gave a formula for the number of such functions.
This work was extended to vector parking functions in \cite{ferreri2025enumeratingvectorparkingfunctions}.

In Theorem~\ref{closed formula FR}, we provide a combinatorial formula for the number of Fubini rankings with $n$ competitors and with a fixed lucky set $I$.

\begin{theorem}\label{closed formula FR} 
Let $I=\{i_1=1, i_2, \ldots, i_k\}\subseteq [n]$ be a lucky set of $\FR_n$, such that $1=i_1<i_2<\cdots < i_k\leq n$, and define $i_{k+1}=n+1$. Then
\begin{equation*}
    |\Lucky_{\FR_n}(I)| = \prod_{\ell=1}^{k} \ell^{i_{\ell+1}-i_{\ell}}.
\end{equation*}
\end{theorem}

\begin{proof}
Recall that the number of lucky cars equals the number of distinct ranks, or equivalently,  the number of blocks in the ordered set partition under the bijection described on  Page~\pageref{BijFR}.   

Let $\alpha = (a_1, a_2, \dots, a_n)$ be a Fubini ranking with $n$ competitors such that $\Lucky(\alpha) = I$. Because $i_1 = 1$ corresponds to a lucky car, all entries $a_2, \dots, a_{i_2 - 1}$ must be equal to $a_1$; that is, all these competitors are tied with the first car, hence they are unlucky cars. The rank value $a_1$ itself can be chosen freely among the $k$ available ranks,   contributing $k  \cdot 1^{i_2 - i_1 - 1}$ possibilities.

Next, $a_{i_2}$ must be distinct from $a_1$ and can therefore be chosen in $k - 1$ ways.  Each of the subsequent entries $a_{i_2+1},\dots,a_{i_3-1}$ corresponds to an unlucky car and may take any of the two values $\{a_1,a_{i_2}\}$, giving 
$(k-1)\,2^{i_3-i_2-1}$ configurations for this segment.

In general, for $1 \leq \ell \leq k$, the element $a_{i_\ell}$ must be distinct from all previous $a_{i_j}$ with $1 \leq j < \ell$. Since $\ell - 1$ ranks have already appeared, $a_{i_\ell}$ has $k - \ell + 1$ possible values.  The entries $a_{i_\ell + 1}, \dots, a_{i_{\ell+1} - 1}$ correspond to unlucky cars,  each of which may take any of the $\ell$ ranks already introduced.   Hence this segment contributes 
$(k-\ell+1)\,\ell^{\,i_{\ell+1}-i_\ell-1}$ possibilities. 

By the multiplication principle, the total number of Fubini rankings with lucky set $I$ is
\[
|\Lucky_{\FR_n}(I)|
   = \prod_{\ell=1}^{k} (k-\ell+1)\,\ell^{\,i_{\ell+1}-i_\ell-1} = \prod_{\ell=1}^{k}\ell^{i_{\ell+1} - i_\ell},
\]
which completes the proof.
\end{proof}

For example, let $I=\{1,2,5\}$. Then, by Theorem~\ref{closed formula FR},
\[
|\Lucky_{\FR_5}(\{1,2,5\})|
   = 1^{\,2-1}\,2^{\,5-2}\,3^{\,6-5}
   = 2^3\cdot3 = 24.
\]
The $24$ Fubini rankings with $5$ competitors whose lucky cars are $\{1,2,5\}$ are 
\begin{align*}
& (1,2,2,2,5),\ (2,1,2,2,5),\ (5,2,2,2,1),\ (2,5,2,2,1),\\
& (2,4,2,4,1), \ (2,4,4,2,1), \ (4,2,4,2,1), \ (4,2,2,4,1),\\
& (1,3,3,3,2), \ (2,3,3,3,1), \ (3,1,3,3,2), \ (3,2,3,3,1),\\
& (1,3,1,3,5), \ (1,3,3,1,5), \ (3,1,1,3,5), \ (3,1,3,1,5),\\
& (1,4,1,4,3), \ (1,4,4,1,3), \ (4,1,1,4,3), \ (4,1,4,1,3),\\
& (1,4,1,1,5), \ (4,1,1,1,5), \ (1,5,1,1,4), \ (5,1,1,1,4).
\end{align*}

\subsubsection{Weakly Increasing Fubini Rankings}\label{sec:weakly increasing fubinis lucky sets}
We now enumerate weakly increasing Fubini rankings with $n$ competitors and with a specified set of lucky cars.
\begin{theorem}\label{closed formula FRW}
Let $I=\{i_1=1, i_2, \ldots, i_k\}\subseteq [n]$ be a lucky set of $\FR^\uparrow_n$, such that $1=i_1<i_2<\cdots < i_k\leq n$. Then there exists a unique $\alpha\in \FR_{n}^\uparrow$ with $\Lucky(\alpha)=I$.
\end{theorem}

\begin{proof}
Set $i_{k+1}=n+1$ and define $\alpha=(a_1,a_2,\dots,a_n)$ by
$a_\ell=i_j$ whenever $i_j \leq \ell < i_{j+1}$. Then $\alpha$ is weakly increasing and constant on each block of consecutive indices 
$\{i_j, i_j+1, \dots, i_{j+1}-1\}$. Moreover, $a_{i_j}=i_j$ for every $j$, so the lucky positions are exactly $I$, while all other entries within each block repeat the same value and are unlucky. The uniqueness follows directly from the construction.  
\end{proof}

 For example, let $I=\{1,3,7\}$. The unique weakly increasing Fubini ranking with $8$ competitors is $\alpha=(1,1,3,3,3,3, 7,7)$. 
    
    \section{Enumeration of Unit Fubini Rankings by Number of Lucky Cars}\label{secB}

 We begin by recalling the following definition (cf.\ \cite{unit_pf,Hadaway_unit_interval}). 

\begin{definition}\label{def:unit interval pfs}
    A \textit{unit interval parking function} is a parking function with the additional restriction that a car must park in either its preferred space or the space in front of it. If a car cannot park in either of those spots, it fails to park. Consequently, a unit interval parking function can contain at most two occurrences of the same preference.
\end{definition}
Let $\UPF_n$ denote the set of unit interval parking function of length $n$. For example, $(1,1,2,3)\in \UPF_4$, where car~1 parks in spot~1, car~2 prefers spot~1 but parks one position ahead (in spot~2), car~3 prefers spot~2 but parks in spot~3, and car~4 takes spot~4.  
Note that $(1,1,3,2)$ is not a unit interval parking function, since car~4 would need to park two positions ahead of its preference.
This example illustrates that, unlike Fubini rankings, unit interval parking functions are not invariant under permutations: a permutation of a unit interval parking function is not necessarily a unit interval parking function. 
Both \cite{unit_pf,Hadaway_unit_interval} prove that $|\UPF_n| = \Fub_n$.

Let 
$f_\UPF(n,k)$ denote the number of unit interval parking functions of length $n$  with $k$ lucky cars.
Given a Fubini ranking $\alpha=(a_1,a_2,\ldots,a_n)$, if $a_{i_1},a_{i_2},\ldots,a_{i_j}$ is the full set of entries tied at rank $x$, the map $\psi:\FR_n\to \UPF_n$ defined in \cite[Lemma~2.6]{unit_pf} assigns
\[a_{i_1},a_{i_2},\ldots,a_{i_j} \mapsto
 x, x, x+1,x+2, \ldots, x+j-1.\]
  Brandt et al. \cite[Theorem 2.5]{unit_pf} established that $\psi$ is a bijection between the set of Fubini rankings with~$n$ competitors and the set of unit interval parking functions of length~$n$.  

Moreover, the map~$\psi$ ensures that the first car with preference~$x$ will be a lucky car, thereby preserving $\Lucky(\alpha) = \Lucky(\psi(\alpha))$. Consequently, the same enumerative results obtained in  \Cref{secA} apply to unit interval parking functions as well.  
In particular, the number of unit interval parking functions of length~$n$ with $k$ lucky cars is
$f_{\UPF}(n,k) = k!\, S(n,k)$.

The  map $\psi$ also preserves the increasing condition. That is, if $\alpha$ is weakly increasing, then $\psi(\alpha)$ is also weakly increasing, and conversely.  This observation, together with the results from \Cref{secA}, immediately yields the same enumeration for weakly increasing unit interval parking functions.  
For example, the number of weakly increasing unit interval parking functions of length~$n$ with $k$ lucky cars is     $f_\UPF^\uparrow(n,k) = f_\FR^\uparrow(n,k) = \binom{n-1}{k-1}$. 

Now we introduce the main object of study in this section.

\begin{definition}[Definition 3.1, \cite{boolean}]\label{def: unit fubini rankings}
    A \textit{unit Fubini ranking} is a Fubini ranking that also satisfies the necessary conditions of a unit interval parking function. Namely, a unit Fubini ranking is a Fubini ranking with at most two instances of the same rank.    We denote the set of all such rankings by $\UFR_n$ and note that  $\UFR_n = \UPF_n \cap \FR_n$.
\end{definition}

For example, $(1,1,4,3)$ is a unit Fubini ranking, since it is both a unit interval parking function and a Fubini ranking.   However, $(1,1,2,3)$ is not a unit Fubini ranking: although it is a unit interval parking function, it is not a Fubini ranking.  
Similarly, $(1,1,1,4)$ is not a unit Fubini ranking, since it is a Fubini ranking but not a unit interval parking function.

Elder, Harris, Kretschmann, and Mart\'inez Mori \cite{boolean} establish a bijection between 
the set of unit Fubini rankings with $n$ competitors and with $n-k$ lucky cars and the set of Boolean intervals of rank $k$ in the left weak Bruhat order of the symmetric group $S_n$. From this bijection, they derived the following enumeration; however, we provide an alternative proof.

\begin{theorem}[Theorems 1.1 and 1.2, \cite{boolean}]\label{thm:unit fubini with k lucky}
Let $f_\UFR(n,k)$ denote the number of unit Fubini rankings with $n$ competitors and $k$ lucky cars. Then
    \[ f_\UFR(n,k) = \frac{n!}{2^{n-k}} \binom{k}{n-k}.\]    
\end{theorem}

\begin{proof}
Let $\alpha\in \UFR_n$ be a unit Fubini ranking with $n$ competitors and with $k$ distinct ranks $1=r_1<r_2< \dots< r_k$. Let $c_i$ denote the number of cars having rank $r_i$, so that  $c_i \in \{1,2\}$ and 
$\sum_{j=1}^kc_j = n$. 

We proceed as follows. From the $n$ available positions, choose $c_1$ of them to place the entries labeled $1$; then choose $c_2$ positions to place the entries labeled $1 + c_1$; and so on. The number of ways to do this is given by the multinomial coefficient
$\binom{n}{c_1, c_2, \ldots, c_k}$.  

Let $i$ be the number of ranks that occur once, and $j$ the number of ranks that occur twice.   Then $i + j = k$ and $i + 2j = n$. Hence, the multinomial coefficient simplifies to
\[
\binom{n}{c_1, c_2, \dots, c_k} = \frac{n!}{1!^i 2!^j} = \frac{n!}{2^j}.
\]
To obtain the total number of unit Fubini rankings with~$n$ competitors and with $k$ distinct ranks,  
we sum over all compositions $(c_1, c_2, \ldots, c_k) \vDash n$ with parts in \( \{1,2\} \). The number of such compositions is given by \( \binom{k}{n-k} \) (cf.~\cite{Heubach2016}). Since $j = n - k$, we conclude that
\[
f_\UFR(n,k) = \sum_{(c_1, c_2, \ldots, c_k) \vDash n} \frac{n!}{2^j} = \frac{n!}{2^j} \binom{k}{n-k}=\frac{n!}{2^{n-k}} \binom{k}{n-k}.\qedhere
\]
\end{proof}

\begin{theorem}\label{recSec3}
For $n, k\geq 1$, we have 
\[f_\UFR(n,k)=k\,f(n-1,k-1)+\bigl(2k-n+1\bigr)\,f(n-1,k).\]
\end{theorem}
\begin{proof}
Let $\alpha$ be a unit Fubini ranking of length $n$ with $k$ lucky cars.  
We analyze the position of the last car, that is, the $n$th car.  There are two possibilities.   If the last car ties with another car, we begin with a unit Fubini ranking of size $n-1$ having $k$ distinct ranks, which can be done in $f_{\UFR}(n-1,k)$ ways.   Let $i$ denote the number of ranks that appear twice, and $j$ the number of ranks that appear once.   Since there are $k$ ranks in total, we have $i + j = k$ and $2i + j = n - 1$. Solving these equations yields $i = n - k - 1$ and $j = 2k - n + 1$.   The value $j$ represents the number of singleton ranks available, and only these can receive the new car to form a tie block.   Therefore, there are $(2k - n + 1) f_{\UFR}(n-1,k)$ rankings of this type.

Otherwise, if the last car is not tied with any other car, then the first $n-1$ cars must form a unit Fubini ranking with $k-1$ distinct ranks, which can be done in $f_{\UFR}(n-1,k-1)$ ways.  
In this case, the last car may take any of the $k$ possible rank positions, giving $k f_{\UFR}(n-1,k-1)$ additional configurations.

Adding both cases, we obtain the desired recurrence relation.
\end{proof}

Using the recurrence relation in Theorem~\ref{recSec3}, we can compute the first few values of the sequence $f_{\UFR}(n,k)$:
$$[f_\UFR(n,k) ]_{n, k\geq 0}=
\begin{pmatrix}
 1 & 0 & 0 & 0 & 0 & 0 & 0 & 0 \\
 0 & 1 & 0 & 0 & 0 & 0 & 0 & 0 \\
 0 & 1 & 2 & 0 & 0 & 0 & 0 & 0 \\
 0 & 0 &  6 & 6 & 0 & 0 & 0 & 0 \\
 0 & 0 & \framebox{\textbf{6}} & 36 & 24 & 0 & 0 & 0 \\
 0 & 0 & 0 & 90 & 240 & 120 & 0 & 0 \\
 0 & 0 & 0 & 90 & 1080 & 1800 & 720 & 0 \\
 0 & 0 & 0 & 0 & 2520 & 12600 & 15120 & 5040 \\
\end{pmatrix}.$$  
This array does not appear in the OEIS. For example, the boxed entry corresponds to $f_\UFR(4,2)=6$.  The six unit Fubini rankings with~$4$ competitors and with exactly two lucky cars are
\[
(1,1,3,3), \ (1,3,1,3), \ (1,3,3,1),  \ (3, 3, 1, 1), \ (3, 1, 3, 1), \ (3,1,1,3).
\]

We can restrict the bijection given on page~\pageref{BijFR} between Fubini rankings and ordered set partitions.

\begin{definition}
A \emph{restricted ordered set partition} of $[n]$ into $k$ blocks is a sequence of disjoint nonempty subsets $(B_1, B_2, \dots, B_k)$ whose union is $[n]$ and satisfying the condition that the size of each block is at most~$2$.
\end{definition}

Let $S_{\leq 2}(n,k)$ denote the number of restricted ordered set partitions of $[n]$ into $k$ blocks.   
This sequence has been extensively studied in the context of set partitions (see, for example, \cite{Caicedo, JungMezoRamirez2018, Moll} and references therein).  

Under the bijection, the number of blocks equals the number of distinct ranks, and hence equals the number of lucky cars.
 
\begin{corollary}\label{coro: k lucky unit fub}
    Let $f_\UFR(n,k)$ denote the number of unit Fubini rankings with $n$ competitors and $k$  lucky cars. Then $f_\UFR(n,k)=k! S_{\leq 2}(n,k)$.
\end{corollary}

Before we present the next result, we prove the following identity which will be useful for \Cref{cor:lucky poly for unit fubini ranking}.

\begin{lemma}\label{prop:general polys}
For $n\geq 0$,
\[\sum_{k=0}^{\lfloor\frac{n}{2}\rfloor}x^k\binom{n-k}{k}=\frac{1}{\sqrt{1+4x}}\left(\left(\frac{1 + \sqrt{1+4x}}{2}\right)^{n+1}-\left(\frac{1 - \sqrt{1+4x}}{2}\right)^{n+1}\right).\]
\end{lemma}

\begin{proof}
Let 
\[
S_n(x) \coloneqq \sum_{k=0}^{\lfloor n/2 \rfloor} x^k \binom{n-k}{k}.
\]
Using the identity  $\binom{n-k}{k} = \binom{n-k-1}{k-1} + \binom{n-k-1}{k}$ and shifting indices, we obtain the recurrence relation
\[
S_n(x) = S_{n-1}(x) + x S_{n-2}(x), \qquad (n \geq 2),
\]
with initial conditions $S_0(x) = 1$ and $S_1(x) = 1$. The characteristic equation associated with this recurrence is  $\lambda^2-\lambda-x=0$ whose roots are
\begin{align*}
\lambda_1=\frac{1 + \sqrt{1+4x}}{2}  \quad \text{and} \quad \lambda_2=\frac{1 - \sqrt{1+4x}}{2}.
\end{align*}
Hence, for some constants $A$ and $B$, the general solution is $S_n(x)=A\lambda_1^n+B\lambda_2^n$. From the initial conditions $S_0(x) = 1$ and $S_1(x) = 1$, we obtain the system
\[
A + B = 1, \qquad A \lambda_1 + B \lambda_2 = 1,
\]
whose solution is 
\begin{align*}
A= \frac{\lambda_1}{\lambda_1-\lambda_2} \quad \text{and} \quad  B=\frac{\lambda_2}{\lambda_2-\lambda_1}. 
\end{align*}

Altogether, we get
\[S_n(x)=\frac{\lambda_1^{n+1}-\lambda_2^{n+1}}{\lambda_1-\lambda_2}=\frac{1}{\sqrt{1+4x}}\left(\left(\frac{1 + \sqrt{1+4x}}{2}\right)^{n+1}-\left(\frac{1 - \sqrt{1+4x}}{2}\right)^{n+1}\right),
\]
as claimed.
\end{proof}

From \Cref{thm:unit fubini with k lucky}, we deduce the following result. 

\begin{corollary}
\label{cor:lucky poly for unit fubini ranking}
The lucky polynomial of unit Fubini rankings with $n$ competitors is given by
\begin{align*}
L_{\UFR_n}(q)&=\sum_{\alpha\in\UFR_n}q^{\lucky(\alpha)}= \frac{n!}{2^{n+1}\sqrt{q(q+2)}}\left((q+\sqrt{q(q+2)})^{n+1}-(q-\sqrt{q(q+2)})^{n+1}\right).
\end{align*}

\end{corollary}
\begin{proof}
We can rewrite the sum as follows,
\begin{align*}
L_{\UFR_n}(q)&=\sum_{\alpha\in\UFR_n}q^{\lucky(\alpha)}=
\sum_{k=0}^{n}\frac{n!}{2^{n-k}}\binom{k}{n-k}q^k\\
&=n!\sum_{k=0}^{\lfloor\frac{n}{2}\rfloor}\frac{1}{2^k}\binom{n-k}{k}q^{n-k}\\
&=q^nn!\sum_{k=0}^{\lfloor\frac{n}{2}\rfloor}\left(\frac{1}{2q}\right)^k\binom{n-k}{k}.
\end{align*}
By \Cref{prop:general polys}, with $x=\frac{1}{2q}$, we obtain
\begin{align*}
L_{\UFR_n}(q)&=q^nn!\sum_{k=0}^{\lfloor\frac{n}{2}\rfloor}\left(\frac{1}{2q}\right)^k\binom{n-k}{k} \\
&= \frac{q^nn!}{\sqrt{1+4(\frac{1}{2q})}}\left(\left(\frac{1 + \sqrt{1+4(\frac{1}{2q})}}{2}\right)^{n+1}-\left(\frac{1 - \sqrt{1+4(\frac{1}{2q})}}{2}\right)^{n+1}\right).
\end{align*}

By simplifying further, we obtain the desired result. 
\end{proof}

The first few polynomials $L_{\UFR_n}(q)$ are
\begin{align*}
&L_{\UFR_0}(q)= 1, \quad  L_{\UFR_1}(q)= q, \quad  L_{\UFR_2}(q)= 2q^2+q, \quad L_{\UFR_3}(q)=6 q^3+6 q^2,\\ &L_{\UFR_4}(q)= 24 q^4+36 q^3+6 q^2, \quad L_{\UFR_5}(q)=120 q^5+240 q^4+90 q^3.
\end{align*}

\begin{theorem}\label{teor: exp number for ufp}
The expected number of lucky cars in a unit Fubini ranking with $n$ competitors is
        \[\mathbb{E}[\lucky(\alpha) : \alpha \in \UFR_n)]= \frac{\left(\left(3n+1\right) \sqrt{3}+6 n \right) \left(1+\sqrt{3}\right)^{n}- \left(\left(3 n +1\right) \sqrt{3}-6 n \right)\left(1-\sqrt{3}\right)^{n}}{3\left(3+\sqrt{3}\right) \left(1+\sqrt{3}\right)^{n}+3 \left(3-\sqrt{3}\right) \left(1-\sqrt{3}\right)^{n}}.
        \]
        Moreover, 
         \[\mathbb{E}[\lucky(\alpha): \alpha \in \UFR_n)]\sim \frac{3 \left(2+\sqrt{3}\right) n+\sqrt{3}}{3 \left(3+\sqrt{3}\right)}.\]
\end{theorem}

\begin{proof}
We proceed by evaluating $L_{\UFR_n}(1)$, and taking the derivative of $L_{\UFR_n}(q)$ with respect to $q$ and evaluating at $q=1$.

Observe that
\begin{align}
    L_{\UFR_n}(1)&=\frac{n!}{2^n}\displaystyle\sum_{k=1}^n  2^k\binom{k}{n-k} \label{Lat1}\intertext{and} 
    \left.\left( \frac{\partial}{\partial q} L_{\UFR_n}(q) \right)\right|_{q=1} &= \frac{n!}{2^n}\displaystyle\sum_{k=1}^n 2^kk\binom{k}{n-k}.    \label{Lprimeat1}
\end{align}
Next, we find closed formulas for the summations in \eqref{Lat1}  and \eqref{Lprimeat1}.

Define 
\begin{align}
    f(n)&\coloneq\displaystyle\sum_{k=1}^n 2^k\binom{k}{n-k} \label{sum1} \intertext{ and }
    g(n)&\coloneq \displaystyle\sum_{k=1}^n 2^kk\binom{k}{n-k} \label{sum2}.
\end{align}
Since both summands are proper hypergeometric terms, we apply Zeilberger’s creative telescoping method \cite{PWZ}. 
Denote the summands on the right-hand side of \Cref{sum1} by $F_1(n,k)$, that is, 
$$F_1(n,k)=2^k\binom{k}{n-k}.$$
By Zeilberger’s algorithm, $F_1(n,k)$ satisfies the recurrence relation
\begin{align}\label{EcDoron}
F_1(n+2,k)-2F_1(n+1,k)-2F_1(n,k)=G_1(n,k+1)-G_1(n,k),
\end{align}
with  certificate 
$$R_1(n,k)=-\frac{(n-2 k) (-2 k+n+1)}{(-k+n+1) (-k+n+2)}.$$
That is, $R_1(n,k) = F_1(n,k)/G_1(n,k)$ is a rational function in both variables.  Summing both sides of \Cref{EcDoron} over~$k$ gives  
\begin{align*}
f(n+2)-2f(n+1)-2f(n)=0, \quad \mbox{with }f(1)=2\mbox{ and } f(2)=6.
\end{align*}
Similarly, let $F_2(n,k)$ denote the summands on the right-hand side of \Cref{sum2}, that is, $$F_2(n,k)=2^kk\binom{k}{n-k}.$$
By Zeilberger’s algorithm, we have that $F_2(n,k)$ satisfies the relation
\begin{multline}\label{EcDoron2}
(-2 + n) (1 + n)F_2(n+2,k)-2 (-5 + n^2)F_2(n+1,k)\\-2(-1 + n) (2 + n)F_2(n,k)=G_2(n,k+1)-G_2(n,k),
\end{multline}
with  certificate 
$$R_2(n,k)=\frac{(n-2 k) (-2 k+n+1) \left(k \left(n^2-n-4\right)+2 (n+2)\right)}{k (-k+n+1)
   (-k+n+2)}.$$
Summing both sides of \Cref{EcDoron2} over~$k$ yields
\begin{align*} 
2(n + 2)(n - 1)g(n) + (2n^2 - 10)g(n+1) - (n - 2)(n + 1)g(n+2)=0,
\end{align*}
 with $g(1)=2$ and $ g(2)=10$.
Solving these recurrences, for $n\geq 1$, we obtain 
\begin{align*}
f(n)&=\frac{\left(3+\sqrt{3}\right) \left(\sqrt{3}+1\right)^{n}+ \left(3-\sqrt{3}\right) \left(1-\sqrt{3}\right)^{n}}{6}
\intertext{and}
g(n)&=\frac{\left(\left(3 n +1\right) \sqrt{3}+6 n \right) \left(1+\sqrt{3}\right)^{n}- \left(\left(3 n +1\right) \sqrt{3}-6 n \right)\left(1-\sqrt{3}\right)^{n}}{18}.
\end{align*}
Combining these closed forms, we have
\[
   \frac{L'_{\UFR_n}(1)}{L_{\UFR_n}(1)} = \frac{g(n)}{f(n)},
\]
which gives the stated expression for the expected number of lucky cars. 

Finally, we can rewrite the expected value as
\[
   \frac{3\sqrt{3}\,n + 6n + \sqrt{3}}{3(3+\sqrt{3})}
   - \frac{2\sqrt{3}\,(n+1)(1-\sqrt{3})^n}
   {(3+\sqrt{3})\bigl(3(1-\sqrt{3})^n - \sqrt{3}(1-\sqrt{3})^n 
      + 3(1+\sqrt{3})^n + \sqrt{3}(1+\sqrt{3})^n\bigr)}.
\]
The second term tends to zero as $n \to \infty$, and therefore
\[
   \frac{L'_{\UFR_n}(1)}{L_{\UFR_n}(1)} = \frac{g(n)}{f(n)} \sim   \frac{3(2+\sqrt{3})\,n + \sqrt{3}}{3(3+\sqrt{3})}.\qedhere
\]
\end{proof}

Next, we give a closed formula for the exponential generating function of $f_{\UFR}(n,k)$.

\begin{theorem}\label{EGFUFR}
The exponential generating function for the number of unit Fubini rankings with $n$ competitors and with $k$ lucky cars is 
  \[\sum_{n\geq 0} \sum_{k\geq 0} f_\UFR(n,k)q^k \frac{x^n}{n!}=\frac{1}{1-q(x+\frac{x^2}{2})}.
\]
\end{theorem}

\begin{proof}
By  \Cref{coro: k lucky unit fub}, we have   $f_{\UFR}(n,k)=k!S_{\leq 2}(n,k)$. The restricted Stirling numbers of the second kind satisfy the exponential generating function (cf.~\cite{JungMezoRamirez2018}) 
\[\sum_{n\geq 0} S_{\leq 2}(n,k)\frac{x^n}{n!}=\frac{1}{k!}\left(x+\frac{x^2}{2}\right)^k.\]
Therefore, 
\begin{align*}
   \sum_{n\geq 0}\sum_{k\geq 0} k!S_{\leq 2}(n,k)q^k\frac{x^n}{n!} &= \sum_{k\geq 0}k!q^k\left(\sum_{n\geq 0}S_{\leq 2}(n,k)\frac{x^n}{n!}\right)
    =\sum_{k\geq 0} k!q^k\frac{1}{k!}\left(x+\frac{x^2}{2}\right)^k\\
    &=\sum_{k\geq 0}\left(x+\frac{x^2}{2}\right)^kq^k
    =\frac{1}{1-q(x+\frac{x^2}{2})}.\qedhere
\end{align*}
\end{proof}

As a direct consequence, setting $q=1$ we obtain  the exponential generating function for the total number of unit Fubini rankings.

\begin{corollary}\label{corosumF}
The exponential generating function for the total number of unit Fubini rankings with~$n$ competitors is
\[
\sum_{n\geq 0} |\UFR_n|\frac{x^n}{n!}
   = \frac{1}{1-x-\frac{x^2}{2}}.
\]
\end{corollary}
From the generating function in \Cref{corosumF} we obtain that the first few values of the sequence $ |\UFR_n|$ which are  (OEIS \seqnum{A080599} \cite{OEIS})
\[1, \quad  1, \quad  3, \quad  12, \quad  66, \quad  450, \quad  3690, \quad  35280, \quad  385560, \quad  4740120, \dots \]

\subsection{Weakly Increasing Unit Fubini Rankings}\label{sec:w_inc_ufr}
Let  $\UFR_n^\uparrow$ be the set of weakly increasing unit Fubini rankings with $n$ competitors. The following result was proved in \cite[Theorem~3.8]{boolean}. We give an alternative proof by establishing a bijection with compositions of $n$ with parts equal to $1$ or $2$.
\begin{theorem}[Theorem 3.8, \cite{boolean}]\label{numWIFR}
    The number of weakly increasing unit Fubini rankings with $n$ competitors is
$|\UFR_n^\uparrow| = F_{n+1}$,
where $F_{n+1}$ denotes the $(n+1)$th Fibonacci number, defined by the recursion $F_{n+1} = F_n + F_{n-1}$ with initial values $F_1 = F_2 = 1$.
\end{theorem}

\begin{proof}
Any weakly increasing unit Fubini ranking $\alpha$ with $n$ competitors corresponds to a unique composition of $n$ with parts equal to 1 or 2. Indeed, since each rank in $\alpha$ appears either once or twice consecutively, we can associate a part of size 1 to each isolated rank and a part of size 2 to each pair of equal consecutive ranks.  This algorithm defines a bijection between weakly increasing unit Fubini rankings and compositions of $n$ with parts in $\{1,2\}$. For example, the image of $(1,1,3,4,4,6,6,7)$ is the composition $(2,1,2,2,1)$. 
\end{proof}
We now give an independent proof of the following result. 

\begin{theorem}[Theorem 1.3, \cite{boolean}]\label{numWIFRb}
Let $f^\uparrow_{\UFR}(n,k)$ denote the number of weakly increasing unit Fubini rankings with $n$ competitors and  $k$ lucky cars. Then     $$f_\UFR^\uparrow(n,k) = \binom{k}{n-k}.$$
\end{theorem}
\begin{proof}
In a unit Fubini ranking with $n$ competitors and with $k$ distinct ranks, exactly $n-k$ ranks must appear twice and the remaining $2k-n$ appear once. Let $i$ be the number of distinct ranks  that appear twice and $j$ the number of distinct ranks that appear once. Since there are $k$ distinct ranks in total, we have $i+j=k$. On the other hand, we have the relation $2i+j=n$. Subtracting these equations yields $i=n-k$ and hence $j=2k-n$. Choosing which of the $k$ ranks are repeated determines the ranking uniquely, and there are $\binom{k}{n-k}$ ways to do so.
\end{proof}

\begin{remark}\label{recWFRI}
It is immediate from Theorem~\ref{numWIFRb} that the numbers 
$f_\UFR^{\uparrow}(n,k)$ satisfy the recurrence relation
\[
f_\UFR^\uparrow(n,k)
= f_\UFR^\uparrow(n-1,k-1)\;+\;f_\UFR^\uparrow(n-2,k-1),
\qquad(n\geq2, k\geq 1)
\]
\end{remark}

The first few values of the sequence $f_\UFR^\uparrow(n,k)$ are shown below:
$$[f_\UFR^\uparrow(n,k) ]_{n, k\geq 0}=
\begin{pmatrix}
 1 & 0 & 0 & 0 & 0 & 0 & 0 & 0 \\
 0 & 1 & 0 & 0 & 0 & 0 & 0 & 0 \\
 0 & 1 & 1 & 0 & 0 & 0 & 0 & 0 \\
 0 & 0 & 2 & 1 & 0 & 0 & 0 & 0 \\
 0 & 0 & 1 & \framebox{\textbf{3}} & 1 & 0 & 0 & 0 \\
 0 & 0 & 0 & 3 & 4 & 1 & 0 & 0 \\
 0 & 0 & 0 & 1 & 6 & 5 & 1 & 0 \\
 0 & 0 & 0 & 0 & 4 & 10 & 6 & 1 \\
\end{pmatrix}.$$  
This array corresponds with OEIS \seqnum{A030528} \cite{OEIS}. For example, the boxed entry corresponds to $f_\UFR(4,3)=3$.  The three weakly increasing unit Fubini rankings with~$4$ competitors and with exactly three lucky cars are
\[
(1,1,3,4), \quad  (1,2,2,4),  \quad (1, 2, 3, 3).
\]

In the following theorem, we derive an explicit expression for the lucky polynomials of weakly increasing unit Fubini rankings:   
\[
L_{\UFR_n^\uparrow}(q) = \sum_{k=0}^n \binom{k}{n-k} q^k,
\]
based on a decomposition of the elements in $\UFR_n^\uparrow$.  Although it is also possible to obtain this formula by applying Zeilberger’s algorithm, as in Theorem~\ref{teor: exp number for ufp}, we prefer to present here a direct combinatorial argument that enriches the exposition.

\begin{theorem}\label{lucpolsec3}
For $n\geq 1$ the lucky polynomial of weakly increasing unit Fubini rankings with $n$ competitors  is given by
\begin{align*}
L_{\UFR_n^\uparrow}(q)&=\sum_{\alpha\in\UFR_n^\uparrow}q^{\lucky(\alpha)}\\
&=\frac{\left(q +\sqrt{q \left(q +4\right)}\right) \left(\frac{2q}{\sqrt{q \left(q +4\right)}-q}\right)^{n}- \left(q-\sqrt{q \left(q +4\right)}\right)\left(\frac{-2q}{q +\sqrt{q \left(q +4\right)}}\right)^{n}}{2 \left(\sqrt{q \left(q +4\right)}\right)},
\end{align*}
  where $L_{\UFR_0^\uparrow}(q)=0$.
\end{theorem}
\begin{proof}
Consider the  bivariate generating function:
\[
F(x,q) \coloneq \sum_{\alpha \in \UFR^\uparrow} x^{|\alpha|} q^{\lucky(\alpha)} = \sum_{n \geq 1} \sum_{\alpha \in \UFR^\uparrow_n} q^{\lucky(\alpha)} x^n,
\]
where \( |\alpha| \) denotes the length of \( \alpha \), and \( \UFR^\uparrow = \bigcup_{n \geq 1} \UFR^\uparrow_n \) is the set of all weakly increasing unit Fubini rankings.

In what follows, we write $\alpha+k$ to mean adding the integer $k$ to every entry in $\alpha$ and whenever we write $x(\alpha)$, with $x$ a word with entries in $\mathbb{N}$, we mean appending each letter in $x$ at the start of the tuple $\alpha$ in the order the letters appear in $x$.
For $n \geq 3$, any $\alpha \in \UFR^\uparrow_n$ can be uniquely decomposed as either $1(\alpha'+1)$, where $\alpha' \in \UFR_{n-1}^\uparrow$, or $11(\alpha''+2)$, where $\alpha'' \in \UFR^\uparrow \cup \{\varepsilon\}$ and $\varepsilon$ denotes the empty ranking (with zero competitors).

In the first case, the initial $1$ contributes one lucky car (weight $q$) and length one (weight $x$), while the remaining part $\alpha'+1$ contributes $F(x,q)$. In the second case, the initial $11$ contributes a single lucky car (again with weight $q$) and length two (weight $x^2$), followed by $\alpha''+2$, which contributes either $F(x,q)$ if $\alpha'' \neq \varepsilon$, or $1$ if $\alpha''$ is empty. Therefore, the contribution to the generating function is
\[
xq F(x,q) + x^2 q (F(x,q) + 1).
\]

For the initial cases $n = 1$ and $n = 2$, we note that $\UFR^\uparrow_1 = \{ (1) \}$ contributes $xq$, and $\UFR^\uparrow_2 = \{ (1,1), (1,2) \}$ contributes $x^2 q + x^2 q^2 = x^2(q + q^2)$. Combining all contributions, we obtain
\[
F(x,q) = xq + x^2(q + q^2) + xq F(x,q) + x^2 q (F(x,q) + 1).
\]
Simplifying this expression, we have \[
F(x,q) = \frac{xq(1 + x)}{1 - xq - x^2 q}.
\]
Since  $F(x,q) = \sum_{n \geq 1} L_{\UFR_n^\uparrow}(q) x^n$, the coefficients satisfy the recurrence relation
\[
L_{\UFR_{n+2}^\uparrow}(q) = q L_{\UFR_{n+1}^\uparrow}(q) + q L_{\UFR_n^\uparrow}(q),
\]
with initial conditions $L_{\UFR_1^\uparrow}(q) = q$ and $L_{\UFR_2^\uparrow}(q) = q + q^2$. Solving this recurrence, for example using Maple, yields the stated closed form.
\end{proof}

\begin{corollary}\label{corosec3b} 
The expected number of lucky cars in a weakly increasing unit Fubini ranking with $n$ competitors is
    $$\mathbb{E}[\lucky(\alpha) :\alpha \in \UFR_n^\uparrow)] = \frac{1}{F_{n+1}} \sum_{k=1}^n k  \binom{k}{n-k}.$$  
    Moreover, 
    \[\mathbb{E}[\lucky(\alpha): \alpha \in \UFR_n^\uparrow)]\sim \frac{1}{10} \left(\left(5+\sqrt{5}\right) n+\sqrt{5}-1\right).\]
\end{corollary}
\begin{proof}
Observe that
\begin{align}
    L_{\UFR_n^\uparrow}(1)&=\sum_{k=1}^n\binom{k}{n-k}=F_{n+1} \label{FiboT}\intertext{and} 
    \left. \left(\frac{\partial}{\partial q}  L_{\UFR_n^\uparrow}(q)\right) \right|_{q=1} &= \sum_{k=1}^nk\binom{k}{n-k}. \label{FiboT2}
\end{align}

For \Cref{FiboT}, note that when $k < \ceil{n/2}$, $\binom{k}{n-k}=0$,which aligns with the fact that there cannot exist unit Fubini rankings with fewer than $\ceil{n/2}$ lucky cars. Therefore, an equivalent summation is
\[
\sum_{k=\ceil{n/2}}^n \binom{k}{n-k} = \sum_{j=0}^{\floor{n/2}} \binom{n-j}{j} = F_{n+1},
\]
which is a well-known identity for the Fibonacci numbers.

The sequence defined in \Cref{FiboT2} corresponds to the OEIS sequence \seqnum{A023610} \cite{OEIS}.  
In particular, 
\begin{align}\label{expression 1}\sum_{k=1}^nk\binom{k}{n-k}=\frac{1}{5}\left((n+1)F_{n+3}+(n-2)F_{n+1} \right), \quad n\geq 1.
\end{align}
Binet's formula for the Fibonacci numbers gives
\begin{align}
\label{expression 2}
F_n\sim \frac{1}{\sqrt{5}}\left(\frac{1+\sqrt{5}}{2} \right)^n.
\end{align}
Combining \Cref{expression 1,expression 2}, we obtain
\begin{align*}
   \frac{L'_{\UFR_n^\uparrow}(1)}{L_{\UFR_n^\uparrow}(1)} &= \frac{1}{5F_{n+1}}\left((n+1)F_{n+3}+(n-2)F_{n+1} \right)\\
   &\sim \frac{3+\sqrt{5}}{10}(n+1) + \frac{1}{5}(n-2)\\
   &=\frac{1}{10} \left(\left(5+\sqrt{5}\right) n+\sqrt{5}-1\right),
\end{align*}
which proves the asymptotic estimate.
\end{proof}

\begin{theorem}\label{EGFWUFR}
The exponential generating function for the number of weakly increasing unit Fubini rankings with $n$ competitors and with $k$ lucky cars is 
  \[\sum_{n\geq 0} \sum_{k\geq 0} f_\UFR(n,k)^\uparrow q^k \frac{x^n}{n!}= \frac{\lambda_1 e^{\lambda_1 x} - \lambda_2 e^{\lambda_2 x}}{\lambda_1 - \lambda_2},
\]
where
$\lambda_{1} = \frac{q + \sqrt{q^{2} + 4q}}{2}$ and $\lambda_{2} = \frac{q - \sqrt{q^{2} + 4q}}{2}$.
\end{theorem}
\begin{proof}
Let 
\[
S_n(q)\coloneqq \sum_{k=1}^{n} \binom{k}{\,n-k\,} q^{k}.
\]
The sequence satisfies the recurrence relation
\[
S_n(q) = q\big(S_{n-1}(q) + S_{n-2}(q)\big), \qquad n \geq 2,
\]
with initial conditions \(S_0(q) = 1\) and $S_1(q) = q$. Let 
\[
A(x,q) = \sum_{n\ge0} S_n(q)\,\frac{x^n}{n!}
\]
be its exponential generating function.  
Multiplying the recurrence by \(x^n/n!\) and summing over \(n \geq 2\) yields the differential equation
\[
A''(x,q)-qA'(x,q)-qA(x,q)=0,
\]
with initial conditions \(A(0,q) = 0\) and \(A'(0,q) = q\).
The characteristic equation \(\lambda^2-q\lambda-q=0\) has roots
\[\lambda_{1}=\frac{q+\sqrt{q^2+4q}}{2} \quad \text{and} \quad \lambda_{1}=\frac{q-\sqrt{q^2+4q}}{2}.\] Hence, for some constants, $A$ and $B$, the general solution is
$A(x,q)=A e^{\lambda_1 x}+B e^{\lambda_2 x}$. From the initial conditions, we obtain the system 
\[
A+B=1,\qquad \lambda_1A+\lambda_2B=q,
\]
whose solution is
\(
A=\frac{\lambda_1}{\lambda_1-\lambda_2},\ 
B=\frac{\lambda_2}{\lambda_2-\lambda_1}.
\)
Therefore,
\[
A(x,q)=\frac{\lambda_1 e^{\lambda_1 x}-\lambda_2 e^{\lambda_2 x}}{\lambda_1-\lambda_2},
\]
as claimed.
\end{proof}

\begin{corollary}\label{EGFWUFRcoro}
The exponential generating function for the total number of weakly increasing unit Fubini rankings with $n$ competitors  is 
  \[\sum_{n\geq 0} |\UFR_n^\uparrow| \frac{x^n}{n!}= \frac{\lambda_1 e^{\lambda_1 x} - \lambda_2 e^{\lambda_2 x}}{\lambda_1 - \lambda_2},
\]
where
$\lambda_{1} = \frac{1 + \sqrt{5}}{2}$ and $\lambda_{2} = \frac{1 - \sqrt{5}}{2}$.
\end{corollary}

\subsection{Unit Fubini Rankings with a Fixed Set of Lucky Cars}\label{sec:unit fubini lucky sets}
The goal of this section is to determine the cardinality of the set of unit Fubini rankings with $n$ competitors having a fixed set of lucky cars. 
Our main result is below.

\begin{theorem} \label{luckySec3a}
Let $I=\{i_1=1, i_2, \ldots, i_k\}\subseteq [n]$ be a lucky set of $\UFR_n$ such that $1=i_1<i_2<\cdots < i_k\le n$.  
Let $U=\{u_1,u_2,\ldots,u_{n-k}\}$ be the ordered set of indices of the unlucky cars, that is, $U=[n]\setminus I$.  
Then
\[
|\Lucky_{\UFR_n}(I)| = k! \prod_{\ell=1}^{n-k}(u_\ell - 2\ell + 1).
\]
\end{theorem}
\begin{proof}
First note that there are $k!$ ways to assign the $k$ distinct ranks to the entries $a_{i_1}, a_{i_2}, \dots, a_{i_k}$ corresponding to the lucky cars.

Recall that in a unit Fubini ranking, each rank may appear at most twice.  
The first unlucky car $u_1$ can therefore take any rank already assigned to a lucky car preceding it.  
The second unlucky car $u_2$ can also take any rank assigned to a preceding lucky car,  
but it cannot share the same value as the previous unlucky car $u_1$.

In general, for $\ell\in [n-k]$, the unlucky car $u_\ell$ can take the rank of any preceding lucky car,  
but it cannot be tied with any of the previous unlucky cars $u_1,\dots,u_{\ell-1}$.  
The number of lucky cars preceding $u_\ell$ is
\[
\max\{s : i_s < u_\ell\} = |\{i\in I : i < u_\ell\}| = (u_\ell - 1) - (\ell - 1) = u_\ell - \ell,
\]
and there are $\ell-1$ preceding unlucky cars.  
Hence, the number of possible choices for car $u_\ell$ is
\[
(u_\ell - \ell) - (\ell - 1) = u_\ell - 2\ell + 1.
\]

Multiplying over all $\ell = 1,2,\dots,n-k$ and accounting for the $k!$ ways of assigning the initial ranks,  
we obtain the desired result.
\end{proof}

\subsubsection{Weakly Increasing  Unit  Fubini Rankings}\label{sec:weakly increasing unit fubini lucky sets}

In this case the collection of lucky sets displays some rather interesting features. We now establish which subsets $I\subseteq[n]$ can be lucky sets for weakly increasing unit Fubini rankings.

\begin{theorem}\label{thm:lucky-characterization}
A subset $I = \{1 = i_1 < i_2 < \cdots < i_k\} \subseteq [n]$ is a lucky set of $\UFR_n^{\uparrow}$ if and only if $|I| \geq \lceil n/2 \rceil$ and $i_{j+1} - i_j \le 2$ for all $j \in [k-1]$.
\end{theorem}

\begin{proof}
Recall that in a weakly increasing unit Fubini ranking of length~$n$, each rank appears at most twice and $a_i \le a_{i+1}$ for all $i \in [n]$.  
From this definition, any two cars sharing the same rank must occupy consecutive positions, which implies $i_{j+1} - i_j \le 2$ for all $j$.  
Moreover, by a previous result, the first car is always lucky.  
Hence, the minimum possible number of lucky cars in $\UFR_n^{\uparrow}$ is $\lceil n/2 \rceil$, completing the proof.
\end{proof}
We conclude with showing that for any lucky set, there is a unique unit Fubini ranking with that lucky set.
\begin{corollary}\label{corosec3cardlu}
For a fixed subset \( I = \{1 = i_1 < i_2 < \cdots < i_k\} \subseteq [n] \) satisfying the characterization in Theorem~\ref{thm:lucky-characterization}, there is exactly one weakly increasing unit Fubini ranking whose lucky set is \( I \).  \end{corollary}
\begin{proof}
Since ranks may repeat only on consecutive positions, the differences \( i_{j+1} - i_j \in \{1,2\} \) (and \( n - i_k \in \{0,1\} \)) uniquely determine the block sizes of equal ranks. Each block receives the ranks \( 1,2,\dots,k \) in order, so the ranking is completely determined.
This establishes the result.
\end{proof}

\section{Future Directions}\label{sec:future}
In this article, we considered the set of Fubini rankings and unit Fubini rankings and enumerated them based on a fixed number or a fixed set of lucky cars, studying also their lucky polynomials. One could repeat this analysis for the following  sets of tuples, which are examples of subsets of ``restricted'' Fubini rankings. 
\begin{enumerate}
\item Brandt et al., \cite[Definition 3.2]{unit_pf} introduced $r$-Fubini rankings of length $n+r$, which are Fubini rankings with $n+r$ competitors whose first $r$ values are distinct. 
\item We recall that $\ell$-interval parking functions are parking functions in which cars park at most $\ell$ spots away from their preference. Note the $1$-interval parking functions are the unit interval parking functions. Aguilar-Fraga et al., \cite[p.\ 16]{lintervalrational} introduced $\ell$-interval Fubini rankings, which are the intersection of $\ell$-interval parking functions of length $n$ which
are also Fubini rankings. 
\item Barreto et al., \cite{BHRRV} introduced a variety of restricted Fubini rankings and restricted unit interval parking functions.
\end{enumerate}
 
\section*{Acknowledgments}
Beerbower, Harris, Martinez, and Shirley acknowledge support by the National Science Foundation under Grant No. DMS-2150434. The authors thank Lybitina Koene for the initial code that facilitated the initial stages of this research.
L.~Martinez was supported by the NSF Graduate Research Fellowship Program under Grant No. 2233066.
P.\ E.\ Harris was supported in part by a Simons Travel Support for Mathematicians award \#00007621.

\bibliographystyle{plain}
\bibliography{bibliography.bib}

\end{document}